\documentclass[11pt,a4paper]{article}

\usepackage{epsf,epsfig,amsfonts,amsgen,amsmath,amstext,amsbsy,amsopn,amsthm
}
\usepackage{ebezier,eepic}
\usepackage{color}
\usepackage{multirow}
\newtheorem{thm}{Theorem}[section]

\newtheorem{lem}{Lemma}[section]
\newtheorem{false statement}{False statement}

\theoremstyle{definition}
\newtheorem{definition}{Definition}

\newtheorem{problem}{Problem}

\newcounter{mathitem}
  {\begin{list}{{$(\roman{mathitem})$}}{
   \setcounter{mathitem}{0}
   \usecounter{mathitem}
   \setlength{\topsep}{0pt plus 2pt minus 0pt}
   \setlength{\parskip}{0pt plus 2pt minus 0pt}
   \setlength{\partopsep}{0pt plus 2pt minus 0pt}
   \setlength{\parsep}{0pt plus 2pt minus 0pt}
   \setlength{\leftmargin}{35pt}
   \setlength{\itemsep}{0pt plus 2pt minus 0pt}}}
  {\end{list}}

\begin{document}

\title{\bf The stability method, eigenvalues and cycles of consecutive lengths}

\date{}

\author{Binlong Li\thanks{Department of Applied Mathematics,
Northwestern Polytechnical University, Xi'an, Shaanxi 710072,
P.R.~China. Email: binlongli@nwpu.edu.cn. Partially supported
by the NSFC grant (No.\ 11601429).}~~~ Bo
Ning\thanks{Corresponding author. College of Computer Science, Nankai University, Tianjin 300071, P.R.
China. Email: bo.ning@nankai.edu.cn. Partially supported
by the NSFC grant (No.\ 11601379).}}
\maketitle

\begin{center}
\begin{minipage}{140mm}
\small\noindent{\bf Abstract:}
Woodall proved that for a graph $G$ of order $n\geq 2k+3$ where $k\geq 0$ is an integer,
if $e(G)\geq \binom{n-k-1}{2}+\binom{k+2}{2}+1$ then $G$ contains
a $C_{\ell}$ for each $\ell\in [3,n-k]$. In this article, we prove a stability
result of this theorem. As a byproduct, we give complete solutions to two problems in \cite{GN19}.
Our second part is devoted to an open problem by Nikiforov:
what is the maximum $C$ such that for all positive $\varepsilon<C$ and sufficiently large $n$, every
graph $G$ of order $n$ with spectral radius $\rho(G)>\sqrt{\lfloor\frac{n^2}{4}\rfloor}$
contains a cycle of length $\ell$ for every $\ell\leq (C-\varepsilon)n$.
We prove that $C\geq\frac{1}{4}$ by a method different from previous ones,
improving the existing bounds. We also derive
an Erd\H{o}s-Gallai type edge number condition for even cycles, which may be of independent
interest.


\end{minipage}
\end{center}

\section{Introduction}
In 1970s, Erd\H{o}s \cite{E71} asked how many edges are needed in a graph on $n$ vertices,
in order to ensure the existence of a cycle of length exactly $n-r$?
Woodall \cite{W76} determined the Tur\'an numbers
of large cycles $C_{\ell}$ for $\ell\in [\lfloor\frac{n+3}{2}\rfloor,n]$
as follows.

\begin{thm}[Woodall \cite{W76}]\label{Thm:BW}
Let $G$ be a graph of order $n\geq 2k+3$ where $k\geq 0$ is an integer.
If $e(G)\geq \binom{n-k-1}{2}+\binom{k+2}{2}+1$, then $G$ contains
a $C_{\ell}$ for each $l\in [3,n-k]$.
\end{thm}

Define $\Gamma$ as a graph which consists of a clique of $n-k-1$ vertices and a clique of $k+2$
vertices sharing one common vertex. The graph $\Gamma$ shows Woodall's theorem is sharp.

In this paper, we shall first consider stability results of Woodall's
theorem following the recent trend.
So it is natural to
recall history of the related stability results of extremal results on cycles.

For non-hamiltonian graphs of order $n$ with given minimum degree,
Erd\H{o}s \cite{E62} proved the following result in 1962.
\begin{thm}[Erd\H{o}s \cite{E62}]
Let $G$ be a graph on $n$ vertices with $\delta(G)\geq k$ where $1\leq k\leq \lfloor\frac{n-1}{2}\rfloor$.
If $G$ is non-hamiltonian then
$$
e(G)\leq\max\left\{\binom{n-k}{2}+k^2,\binom{n-\lfloor\frac{n-1}{2}\rfloor}{2}+{\left\lfloor\frac{n-1}{2}\right\rfloor}^2\right\}.
$$
\end{thm}

As a key lemma to attack the following problem: Among all non-hamiltonian
graphs of order $n$ which have minimum degree
at least $k$, characterize the class of graphs which attain the maximum
spectral radius, the authors \cite{LN16} proved a stability result of Erd\H{o}s' theorem.
This result was also proved by F\"{u}redi, Kostochka and Luo \cite{FKL17}, independently.

\begin{thm}[Li and Ning \cite{LN16}, F\"{u}redi, Kostochka and Luo \cite{FKL17}]
Let $G$ be a graph of order $n\geq 6k+5$. If $\delta(G)\geq k\geq 1$
and
$$
e(G)>\binom{n-k-1}{2}+(k+1)^2,
$$
then $G$ is hamiltonian, unless $G$ is a subgraph of $K_k\vee (kK_1+K_{n-2k})$
or a subgraph of $K_{1}\vee (K_{n-k-1}+K_{k})$.
\end{thm}

In 1977, Kopylov \cite{K77} determined a sharp edge condition for the circumference
of a 2-connected graph. In 2016, F\"{u}redi, Kostochka, and
Verstra\"{e}te \cite{FKV16} proved a stability version of Erd\H{o}s-Gallai
theorem, and finally (together with Luo) \cite{FKLV18} completed
the stability version of Kopylov's theorem \cite{K77}.
In fact, Kopylov's theorem is a special case of a conjecture due to Woodall \cite{W76}, which
refers to the sharp edge condition for circumference of a 2-connected graph
with given minimum degree.
Recently, Ma and Ning \cite{MN20} proved a stability version of Woodall's conjecture.

In this paper, we shall prove a stability result of Theorem \ref{Thm:BW}.
Let us introduce some notation.
\begin{definition}
Let $k$ and $n\geq k+1$ be integers. We define $\mathcal{F}_{n,k}$
to be a family of graphs, such that a graph $G\in\mathcal{F}_{n,k}$
if and only if $G$ is a graph of order $n$ in which there is a
subgraph $K\cong K_{n-k}$, and for each component $H$ of $G-V(K)$,
$V(H)$ is a clique and all vertices in $H$ are adjacent to a same
vertex in $K$. Specially, the graph $L_{n,k}\cong
K_1\vee(K_{n-k-1}+K_k)$ is the one in $\mathcal{L}_{n,k}$ with
maximum number of edges.
\end{definition}
\begin{thm}\label{Thm:LiNing-stability}
Let $G$ be a graph of order $n\geq \max\{6k+17,\frac{(k+4)(k+5)}{2}\}$
where $k\geq 0$. If $$e(G)\geq\binom{n-k-2}{2}+\binom{k+3}{2},$$ then
$G$ is weakly pancyclic with girth 3. Suppose that
$G$ contains no $C_{n-k}$. Then one of the following holds:\\
(a) $G\subseteq L$ for some $L\in\mathcal{L}_{n,k+1}$;\\
(b) $G=L_{n,k+2}\cong K_1\vee (K_{n-k-3}+K_{k+2})$;\\
(c) $k=0$ and $G\subseteq\varGamma_{n,2}:=K_2\vee (K_{n-4}+2K_1)$;\\
(d) $k=1$ and $G=\varGamma_{n,3}:=K_2\vee (K_{n-5}+3K_1)$.
\end{thm}
As a non-trivial byproduct, we give a solution to the following open
problems proposed in \cite{GN19}. By $\rho(G)$ and $q(G)$ we
denote the spectral radius and signless Laplacian spectral radius of
the graph $G$.

\begin{problem}[\cite{GN19}]\label{Prob1}
Let $G$ be a connected graph of order $n$ and $k\geq 1$ be an
integer, where $n$ is sufficiently large compared to $k$.\\
(a) Suppose that $\rho(G)>\rho(L_{n,k})$.
Does $G$ contain a $C_{n-k+1}$?\\
(b) Suppose that $q(G)>q(L_{n,k})$. Does $G$ contain a $C_{n-k+1}$?
\end{problem}
Our answer is the following. When $k=2$, it implies all results in \cite{GN19}.

\begin{thm}\label{Thm:Spectrallargecylce}
Let $k\geq 1$ be an integer. Let $G$ be a graph of order $n$. If either\\
(a) $\rho(G)\geq\rho(L_{n,k})$ where $n\geq \max\{6k+11,\frac{(k+3)(k+4)}{2}\}$ or,\\
(b) $q(G)\geq q(L_{n,k})$ where $n\geq \max\{6k+11,k^2+2k+3\}$,\\
then $G$ contains a $C_{\ell}$ for each $\ell\in [3,n-k+1]$, unless
$G=L_{n,k}$.
\end{thm}

Our technique is to combine the stability methods in extremal graph theory
with spectral technique. Compared with the original method in \cite{LN16},
we need to find such a stability result of number of edges for
$\Omega(\sqrt{n})$ cycles of consecutive lengths,
which is the main new point.

Our second part is devoted to an open problem on cycles with consecutive lengths
due to Nikiforov \cite{N08}.

Bondy \cite{B71} proved that every
hamiltonian graph $G$ on $n$ vertices contains cycles of all lengths
$\ell \in [3,n]$ if $e(G)\geq \frac{n^2}{4}$, unless $n$ is even and $G$ is isomorphic to $K_{\frac{n}{2},\frac{n}{2}}$.
If one drops the condition that ``$G$ is hamiltonian" in Bondy's theorem, a theorem in Bollob\'{a}s'
textbook \cite[Corrolary~5.4]{B76} states such a graph
$G$ contains all cycles $C_{\ell}$ for
each $\ell \in [3,\left\lfloor\frac{n+3}{2}\right\rfloor]$.
Nikiforov \cite{N08} considered cycles of consecutive lengths
from a spectral perspective.

\begin{problem}[Nikiforov \cite{N08}]\label{Prob:2}
What is the maximum $C$ such that for all positive $\varepsilon<C$ and sufficiently large $n$, every
graph $G$ of order $n$ with $\rho(G)\geq \sqrt{\lfloor\frac{n^2}{4}\rfloor}$
contains a cycle of length $\ell$ for every $\ell\leq (C-\varepsilon)n$.
\end{problem}
One may guess $C=\frac{1}{2}$. However, the class of graphs $G=K_s\vee (n-s)K_1$ where
$s=\frac{(3-\sqrt{5})n}{4}$ (see \cite{N08}) shows
$C\leq \frac{(3-\sqrt{5})}{2}$.
Nikiforov \cite{N08} proved that $C\geq \frac{1}{320}$.
Ning and Peng \cite{NP20} slightly refined this as $C\geq \frac{1}{160}$.
Only very recently, Mingqing Zhai and Huiqiu Lin (private communication)
improved these results to $C\geq\frac{1}{7}$.

The second purpose of this article is to show that $C\geq \frac{1}{4}$
by completely different methods.

\begin{thm}\label{Thm:SpectraConsecutiveCycles}\footnote{If $0<\varepsilon<10^{-6}$,
then we can choose $N=2.5\times 10^{10}{\varepsilon}^{-1}$.}
Let $\varepsilon$ be real with $0<\varepsilon<\frac{1}{4}$. Then there
exists an integer $N:=N(\varepsilon)$, such that if $G$ is a graph on
$n$ vertices with $n\geq N$ and $\rho(G)>\sqrt{\lfloor\frac{n^2}{4}\rfloor}$,
then $G$ contains all cycles $C_{\ell}$ with $\ell \in [3,(\frac{1}{4}-\varepsilon)n]$.
\end{thm}

Let $G$ be a graph. We use $\omega(G)$ to denote clique number
of $G$. Let $G_1$ and $G_2$ be two vertex-disjoint graphs.
The \emph{union} of $G_1$ and $G_2$, denoted by $G_1+G_2$, is defined to be a graph
$G'$ with $V(G')=V(G_1)\cup V(G_2)$ and $E(G')=E(G_1)\cup E(G_2)$. The \emph{join} of $G_1$ and $G_2$, denoted by $G_1\vee G_2$,
is a new graph obtained from $G_1+G_2$ by adding all possible edges from
$G_1$ to $G_2$.

Let $A(G)$ be the adjacency matrix of a graph $G$ and $D$ be the degree matrix
of $G$. The \emph{spectral radius} of $G$, denoted by $\rho(G)$, is the largest eigenvalue
of $A(G)$. The \emph{signless Laplacian spectral radius} of $G$, denoted by $q(G)$, is the largest eigenvalue
of the signless Laplacian matrix $Q(G):=A(G)+D(G)$.

The paper is organized as follows. In Section \ref{Sec:Woodall}, we prove
a sharp version of Woodall's theorem and also a stability version of it.
In Section \ref{Sec:Spectral}, we answer Problem \ref{Prob1}
completely. 
In Section \ref{Sec:Nikiforov}, we consider Nikiforov's open problem
on cycles with consecutive lengths. In the last section,
we mention some related problem for further study.

\section{Woodall's theorem updated}\label{Sec:Woodall}
We first refine Woodall's Theorem on Tur\'an number of large cycles
as follows. We call a graph weakly pancyclic if it contains all cycles of lengths
from the smallest one to the largest one.

\begin{thm}\label{Thm:RefinedBW}
Let $G$ be a graph of order $n\geq \max\{6k+11,\frac{(k+3)(k+4)}{2}\}$, where $k\geq 0$.
If $$e(G)\geq \binom{n-k-1}{2}+\binom{k+2}{2},$$  then
$G$ is weakly pancyclic with girth 3. Furthermore,
one of the following is true:\\
(a) $G$ contains a $C_{\ell}$ for each $\ell\in [3,n-k]$;\\
(b) $G=L_{n,k+1}\cong K_1\vee(K_{n-k-2}+K_{k+1})$.
\end{thm}

The proof of Theorem \ref{Thm:RefinedBW} needs the following three lemmas.
The \emph{circumference} of $G$, denoted by $c(G)$, is the length
of a longest cycle in $G$. The \emph{$n$-closure} $cl_n(G)$, is defined to be a graph of order $n$
by recursively joining any pair of non-adjacent vertices with degree sum
at least $n$ till there is no such pair.
\begin{lem}[Bondy and Chv\'atal \cite{BC76}]\label{Lem:BonChv}
Let $G$ be a graph of order $n$ and $C':=cl_n(G)$.
Then $c(G)=c(cl_n(G))$.
\end{lem}

\begin{lem}[\rm Bondy \cite{B71}]\label{Lem:Bondy}
Let $G$ be a graph of order $n$. If $c(G)=c$ and $e(G)>\frac{c(2n-c)}{4}$,
then $G$ is weakly pancyclic with girth 3.
\end{lem}
For the last lemma, its original form
in \cite{LN16} needs the condition ``$k\geq 1$''.
Here we prove the small case that $k=0$. This lemma is the key tool for our proof.
\begin{lem}\label{Lem:LiNing}
Let $G$ be a graph of order $n\geq 6k+5$, where $k\geq 0$. If
$G=cl_n(G)$ and $e(G)>\binom{n-k-1}{2}+(k+1)^2$
then $\omega(G)\geq n-k$.
\end{lem}
\begin{proof}
Recall that the case of $k\geq 1$ was proved in \cite{LN16}. Now set $k=0$.
Suppose that there exist two vertices $x,y\in V(G)$ such that $d(x)+d(y)\leq n-1$.
Let $H:=G-\{x,y\}$. Then $e(G)\leq e(H)+d(x)+d(y)\leq \binom{n-2}{2}+n-1=\binom{n-1}{2}+1$,
a contradiction. Thus, for any two nonadjacent vertices, the degree sum
of them is at least $n$. By the definition of $n$-closure, $G=K_n$ and so $\omega(G)=n$.
\end{proof}
We are in stand for proving Theorem \ref{Thm:RefinedBW}.

\noindent {\bf Proof of Theorem \ref{Thm:RefinedBW}.} Suppose that
$G$ is a graph satisfying the condition. We first show that $G$ is
weakly pancyclic with girth 3. Let $c:=c(G)$. By Lemma
\ref{Lem:Bondy}, we only need to show that
$\binom{n-k-1}{2}+\binom{k+2}{2}>\frac{c(2n-c)}{4}$. If not, then we
have
$$\frac{nc}{2}-\frac{c^2}{4}\geq\frac{n^2-(2k+3)n}{2}+(k+1)(k+2),$$
which implies that $$c^2-2nc+2(n^2-(2k+3)n)+4(k+2)(k+1)\leq 0.$$
However, the discriminant of quadratic form
$$\varDelta=(2n)^2-4\left(2(n^2-(2k+3)n)+4(k+1)(k+2)\right)<0$$ for
$n\geq 2k+5$, a contradiction. This proves the first part of the
theorem.

Now let $G'=cl_n(G)$. Since
$$e(G')\geq e(G)\geq\binom{n-k-1}{2}+\binom{k+2}{2}\geq
\binom{n-k-2}{2}+(k+2)^{2}+1$$ for $n\geq \max\{6k+11,
\frac{(k+3)(k+4)}{2}\}$, by Lemma \ref{Lem:LiNing}, $\omega(G')\geq
n-k-1$. This implies that $c(G')\geq n-k-1$. If $c(G')\geq n-k$,
then $c(G)=c(G')\geq n-k$ by Lemma \ref{Lem:BonChv}. Recall that $G$
is weakly pancyclic, implying that (a) holds. So assume that
$c(G')\leq n-k-1$. Since $c(G')\geq \omega(G')$, we have
$\omega(G')=n-k-1$.

Let $S$ be a clique of $G'$ with $|S|=n-k-1$, let $K=G'[S]$ and
$H=G-S$. Thus $K$ is complete. Let $H_1$ be an arbitrary component
of $H$. If $|N_{G'}(H_1)\cap S|\geq 2$, then clearly $c(G')\geq
n-k$, a contradiction. Thus we conclude that $|N_{G'}(H_1)\cap
S|\leq 1$ for every component $H_1$ of $H$. Specially, every vertex
$v\in V(H)$ has $|N_{G'}(v)\cap S|\leq 1$. Now
\begin{align*}
e(G'-S) & =e(G')-e(K)-e_{G'}(S,V(H))\geq e(G)-e(K)-e_{G'}(S,V(H))\\
        & \geq\binom{n-k-1}{2}+\binom{k+2}{2}-\binom{n-k-1}{2}-(k+1)=\binom{k+1}{2}.
\end{align*}
Since $|V(H)|=k+1$, we infer that $V(H)$ is a $(k+1)$-clique and
equality holds in the above formula. This implies that $G=G'$ and
$|N_G(v)\cap S|=1$ for every $v\in V(H)$. Recall that $|N(H)\cap
S|=1$. All vertices in $H$ have a common neighbor in $S$. We obtain
that $G=L_{n,k+1}$, and (b) holds. The proof is complete.
\hfill{\rule{4pt}{8pt}}

\medskip

We further prove a stability result of Theorem \ref{Thm:BW} as
follows.

\noindent {\bf Proof of Theorem \ref{Thm:LiNing-stability}.} The
argument used here is similar to  Theorem \ref{Thm:RefinedBW}.
However, more details are needed. We first claim that $G$ is weakly
pancyclic with girth 3. By Lemma \ref{Lem:Bondy}, we shall show that
$\binom{n-k-2}{2}+\binom{k+3}{2}>\frac{c(2n-c)}{4}$. Suppose to the
contrary that $c^2-2nc+2(n^2-(2k+5)n)+4(k+2)(k+3)\leq 0$. However,
$$(2n)^2-4\left(2(n^2-(2k+5)n)+4(k+2)(k+3)\right)<0$$ when $n\geq
2k+7$, a contradiction. This proves the first part of the theorem.

Let $G':=cl_n(G)$. If $c(G')\geq n-k$, then by Lemma
\ref{Lem:BonChv}, $c(G)=c(G')\geq n-k$. Recall that $G$ is weakly
pancyclic, implying that $G$ contains $C_{n-k}$. So we assume that
$c(G')\leq n-k-1$. Since $$e(G')\geq e(G)\geq
\binom{n-k-2}{2}+\binom{k+3}{2}\geq \binom{n-k-3}{2}+(k+3)^2+1$$ for
$n\geq \frac{(k+4)(k+5)}{2}$. By Lemma
\ref{Lem:LiNing}, $\omega(G')\geq n-k-2$ for $n\geq 6k+17$. If $\omega(G')\geq n-k$.
then $c(G')\geq\omega(G')\geq n-k$, a contradiction. Now we assume
that $\omega(G')=n-k-2$ or $\omega(G')=n-k-1$. Let $S$ be a clique
of $G'$ with $|S|=\omega(G')$, $K=G'[S]$ and $H=G'-S$.

\underline{Case A. $\omega(G')=n-k-1$.} Let $H_1$ be an arbitrary
component of $H$. If $|N_{G'}(H_1)\cap S|\geq 2$, then $c(G')\geq
n-k$ (recall that $S$ is a clique of $G'$), a contradiction. Thus,
every component $H_1$ of $G'-S$ satisfies
$|N_{G'}(H_1)\cap S|\leq 1$. It follows $G\subseteq G'\subseteq
F\in\mathcal{F}_{n,k+1}$ for some $F$, and (a) holds.

\underline{Case B. $\omega(G')=n-k-2$.} Set $T=\{v\in V(H):
|N_{G'}(v)\cap S|\geq 2\}$. We distinguish the following subcases.

\underline{Case B.1. $|T|=0$.} In this case, every vertex $v\in
V(H)$ has $|N_{G'}(v)\cap S|\leq 1$. Now
\begin{align*}
e(G'-S) & =e(G')-e(K)-e_{G'}(S,V(H))\geq e(G)-e(K)-e_{G'}(S,V(H))\\
        & \geq\binom{n-k-2}{2}+\binom{k+3}{2}-\binom{n-k-2}{2}-(k+2)=\binom{k+2}{2}.
\end{align*}
Since $|V(H)|=k+2$, we infer that $V(H)$ is a $(k+2)$-clique and
equality holds in the above formula. This implies that $G=G'$ and
$|N_G(v)\cap S|=1$ for every $v\in V(H)$. If $|N(H)\cap S|\geq 2$,
then clearly $c(G)\geq n-k$, a contradiction. This implies that all
vertices in $H$ have a common neighbor in $S$. We obtain that
$G=L_{n,k+2}$, and (b) holds.

\underline{Case B.2. $|T|=1$.} Let $v_1$ be the unique vertex in
$T$. Let $H_1$ be an arbitrary component of $H-v_1$. If $v_1\in
N_{G'}(H_1)$, then $N_{G'}(H_1)\cap S=\emptyset$; for otherwise
$c(G')\geq n-k$. Furthermore, If $|N_{G'}(H_1)\cap S|\geq 2$, then
there are two independent edges between $S$ and $V(H_1)$ (notice
that in $G'$, every vertex in $H_1$ has at most 1 neighbors in $S$),
implying that $c(G')\geq n-k$, a contradiction. Thus,
$|N_{G'}(H_1)\cap(S\cup\{v_1\})|\leq 1$ for every component $H_1$ of
$G'-(S\cup\{v_1\})$. This implies that $G\subseteq G'\subseteq
F\in\mathcal{F}_{n,k+1}$ and (a) holds.

\underline{Case B.3. $|T|\geq 2$.} Let $v_1$ be a vertex in $T$ and
$u_1,u_2$ be two vertices in $N_{G'}(v_1)\cap S$. For any other
vertex $v_2\in T$, we have that $N_{G'}(v_2)\cap S=\{u_1,u_2\}$, for
otherwise $c(G')\geq n-k$. Furthermore, $N_{G'}(v_1)=\{u_1,u_2\}$.
In brief, we have $N_{G'}(T)\cap S=\{u_1,u_2\}$. If there are
two vertices in $T$ which are adjacent in $G'$, then $c(G')\geq n-k$, a
contradiction. So $T$ is independent in $G'$. For any
vertex $v\in V(G)\backslash(S\cup T)$, we claim that
$|N_{G'}(v)\cap(S\cup T)|\leq 1$. Indeed, as $v\notin T$, $v$ cannot
have two neighbors in $S$. If $N_{G'}(v)$ contains two vertices in
$T$ or contains one vertex in $T$ and one vertex in $S$, then we
have $c(G')\geq n-k$, a contradiction. Set $t=|T|$. Notice that
$2\leq t\leq k+2$. Now
\begin{align*}
e(G')   & =e(K)+e_{G'}(S,T)+e_{G'}(S\cup T,V(G)\backslash(S\cup T))+e(H-T)\\
        & \leq\binom{n-k-2}{2}+2t+(k+2-t)+\binom{k+2-t}{2}\\
        & =\binom{n-k-2}{2}+\binom{k+3}{2}+\frac{t^2-(2k+1)t}{2}\\
        &\leq  e(G)+\frac{t(t-2k-1)}{2}.
\end{align*}
This implies that $t\geq 2k+1$. Combining with $2\leq t\leq k+2$, it
can only be that $k=0$ and $t=2$, or $k=1$ and $t=3$. In each case $V(G)=S\cup
T$. For the first case, we have $G\subseteq G'=\varGamma_{n,2}$, and (c)
holds. For $k=1$ and $t=3$, $G'=\varGamma_{n,3}$. Moreover, equality
holds in the above inequalities, implying that $G=G'$ and (d) holds.
\hfill{\rule{4pt}{8pt}}
\section{Spectral results}\label{Sec:Spectral}

Let $G$ be a graph and $u,v\in V(G)$. We use $G[u\rightarrow v]$ to denote
a new graph obtained from $G$, by replacing all edges $uw$ by $vw$,
where $w\in N_G(u)\backslash (N_G(v)\cup \{v\})$.
Following Brouwers' book, we call this as ``Kelmans operation''.

In this article, we need some results on spectral properties of graphs under Kelmans operation.
These theorems will play important roles in our answers to Problem \ref{Prob1}.
\begin{thm}[Csikv\'ari \cite{C09}]\label{Thm:C09}
Let $G$ be a graph and $u,v\in V(G)$. Let $G':=G[u\rightarrow v]$.
Then $\rho(G')\geq \rho(G)$.
\end{thm}

\begin{thm}[Li and Ning \cite{LN16}]\label{Thm:LN}
Let $G$ be a graph and $u,v\in V(G)$. Let $G':=G[u\rightarrow v]$.
Then $q(G')\geq q(G)$.
\end{thm}

The following spectral inequalities help us to invert our problems
into ones in extremal style.

\begin{thm}[Hong \cite{H93}]\label{Thm:Hong}
Let $G$ be a graph on $n$ vertices and $m$ edges. If $\delta(G)\geq 1$
then $\rho(G)\leq \sqrt{2m-n+1}$.
\end{thm}

\begin{thm}[Das \cite{Das}]\label{Thm:Das}
Let $G$ be a graph on $n$ vertices and $m$ edges.
Then
$q(G)\leq \frac{2m}{n-1}+n-2$.
\end{thm}

The following two lemmas will be used to determine the extremal graphs.

\begin{lem}\label{Lem:extremal}
Let $G$ be a graph. Suppose that $G$ is a subgraph of
a member in $\mathcal{F}_{n,k}$, where $n\geq 2k+1$.\\
(a) If $\rho(G)\geq \rho(L_{n,k})$, then $G=L_{n,k}$.\\
(b) If $q(G)\geq q(L_{n,k})$, then $G=L_{n,k}$.
\end{lem}

\begin{proof}
(a) Let $F\in\mathcal{F}_{n,k}$ with $G\subseteq F$. Since
$G\subseteq F$, $\rho(G)\leq\rho(F)$, with equality if and only
if $G=F$ (recall that $F$ is connected). Let $K$
be the complete subgraph of $F$ with $|K|=n-k$.
Let $H_1,H_2,\ldots,H_t$ be the components of $F-K$, and let $v_i$,
$i\in[1,t]$, be the unique vertex in $N(H_i)\cap V(K)$. By
a series of Kelmans operation from $v_i$ to $v_1$ for all $v_i\neq
v_1$, we get a graph $F'$ which is a subgraph of $L_{n,k+1}$. By
Theorem \ref{Thm:C09},
$$\rho(G)\leq\rho(F)\leq\rho(F')\leq\rho(L_{n,k}),$$
equality holds if and only if $G=F=F'\cong L_{n,k}$. This proves the
statement (a).

(b) The proof is almost the same as the one of (a). We just use Theorem
\ref{Thm:LN} instead of Theorem \ref{Thm:C09} in the whole proof.
We omit the details.
\end{proof}

\begin{lem}\label{Lem:compared}
Let $n,k$ be integers where $k\geq 1$. Then\\
(a) $\rho(L_{n,k})>\rho(L_{n,k+1})$  for $n\geq 2k+4$;
$\rho(F_{n,1})>\rho(\varGamma_{n,2})$ for $n\geq 6$;
$\rho(F_{n,2})>\rho(\varGamma_{n,3})$ for $n\geq 4$.\\
(b) $q(L_{n,k})>q(L_{n,k+1})$ for $n\geq 2k+4$;
$q(F_{n,1})>q(\varGamma_{n,2})$ for $n\geq 6$;
$q(F_{n,2})>q(\varGamma_{n,3})$ for $n\geq 1$.
\end{lem}
\begin{proof}
(a) Let $V(L_{n,k+1})=X\cup Y\cup \{z\}$, where $X\cup \{z\}$ is the $(k+2)$-clique
in $L_{n,k+1}$ and $Y\cup \{z\}$ is the $(n-k-1)$-clique in $L_{n,k+1}$.
Choose $x\in X$. $L_{n,k}$ can be obtained from $L_{n,k+1}$
by deleting all edges $xx'$ for $x'\in X$ and adding all edges $xy'$ for $y'\in Y$.

Let $M$ be the Perron vector with respect
to $\rho(L_{n,k+1})$, where $x,y,w$ correspond to the eigencomponent of vertices
in $X$, the vertices in $Y$ and the vertex $z$. Let $\rho_1:=\rho(L_{n,k+1})$.
By eigenequation, we have $\rho_1x=kx+z$ and $\rho_1y=(n-k-3)y+z$. It follows
that $(\rho_1-k)x=(\rho_1-(n-k-3))y$. Since $n\geq 2k+4$, we have $y>x$.
Then by Rayleigh quoit,
we have
\begin{align*}
\rho(L_{n,k})-\rho(L_{n,k+1})&\geq 2(n-k-2)xy-2kx^2=2x((n-k-2)y-kx>0.
\end{align*}
This proves $\rho(L_{n,k})>\rho(L_{n,k+1})$ for $n\geq 2k+4$.

Let $M'$ be the Perron vector with respect
to $q(L_{n,k+1})$, where $x,y,w$ correspond to the eigencomponent of vertices
in $X$, the vertices in $Y$ and the vertex $z$. Let $q_1:=q(L_{n,k+1})$.
By eigenequation, we have $q_1x=(2k+1)x+z$ and $\rho_1y=(2n-2k-5)y+z$. It follows
that $(q_1-(2k+1))x=(q_1-(2n-2k-5))y$. If $n\geq 2k+4$, then $y>x$.
Then by Rayleigh quoit,
we have
\begin{align*}
q(L_{n,k})-q(L_{n,k+1})&\geq (n-k-2)(x+y)^2-k(x+x)^2>0.
\end{align*}
This proves $q(L_{n,k})>q(L_{n,k+1})$ for $n\geq 2k+4$.

(b) $\rho(\varGamma_{n,2})\leq \sqrt{2e(\varGamma_{n,2})-n+1}=\sqrt{n^2-6n+15}<n-2=\rho(K_{n-1})<\rho(F_{n,1})$
for $n\geq 6$.
$q(\varGamma_{n,2})\leq \frac{2e(\varGamma_{n,2})}{n-1}+n-2\leq 2(n-2)=q(K_{n-1})<q(F_{n,1})$ for $n\geq 6$.

(c) $\rho(\varGamma_{n,3})\leq \sqrt{2e(\varGamma_{n,3})-n+1}=\sqrt{n^2-8n+25}<n-3=\rho(K_{n-2})<\rho(F_{n,2})$
for $n\geq 4$.
$q(\varGamma_{n,2})\leq \frac{2e(\varGamma_{n,2})}{n-1}+n-2\leq 2(n-3)=q(K_{n-2})<q(F_{n,2})$.
\end{proof}

\noindent
{\bf Proof of Theorem \ref{Thm:Spectrallargecylce}.} If $G$ is
disconnected, then we can add some edges between
different components recursively, and get a connected graph $G'$ with
$\rho(G')>\rho(G)$ and $q(G')>q(G)$. Since the added edges are
not contained in any cycle, if $G'$ contains some cycles, then so
does $G$. Thus we only deal with the case that $G$ is connected.

Suppose that (a) holds. Furthermore, suppose that $G$
does not contains a $C_{\ell}$ for every $\ell\in [3,n-k+1]$.
We shall show that $G=L_{n,k}$.

By Theorem \ref{Thm:Hong}, we have
$$\sqrt{2e(G)-n+1}\geq \rho(G)\geq \rho(L_{n,k+1})\geq n-k-1.$$
It follows that $2e(G)\geq(n-k-1)^2+n-1$. Note that
$$\frac{(n-k-1)^2+n-1}{2}\geq \binom{n-k-1}{2}+\binom{k+2}{2}$$
for $n\geq \frac{(k+2)^2}{2}$. By Theorem \ref{Thm:LiNing-stability},
$G$ is weakly pancyclic with girth 3 for
$n\geq \max\{6k+11,\frac{(k+3)(k+4)}{2}\}$. Furthermore, if $G$ does not
contain a $C_{n-k+1}$, then one of the following is true: (1)
$G\subseteq F$ for some $F\in \mathcal{F}_{n,k}$; (2) $G=L_{n,k+1}$;
(3) $k=1$ and $G\subseteq\varGamma_{n,2}$, or $k=2$ and
$G\subseteq\varGamma_{n,3}$.
By Lemma \ref{Lem:extremal} and Lemma \ref{Lem:compared}, $G=L_{n,k}$.

Suppose that (b) holds. By Theorem \ref{Thm:Das}, we obtain
$$\frac{2e(G)}{n-1}+n-2\geq q(G)\geq q(F_{k+1})\geq 2(n-k-1),$$
which implies that $e(G)\geq \frac{n^2-(2k+1)n+2k}{2}$. Note that
$\frac{n^2-(2k+1)n+2k}{2}\geq \binom{n-k-1}{2}+\binom{k+2}{2}$ for
$n\geq k^2+2k+2$. By Theorem \ref{Thm:LiNing-stability}, $G$ is weakly
pancyclic with girth 3. Furthermore, if $G$ does not contain a
$C_{n-k+1}$, then one of the following is true: (1) $G\subseteq F$
for some $F\in \mathcal{F}_{n,k}$; (2) $G=L_{n,k+1}$; (3) $k=1$ and
$G\subseteq\varGamma_{n,2}$, or $k=2$ and
$G\subseteq\varGamma_{n,3}$. By Lemma \ref{Lem:extremal}
and Lemma \ref{Lem:compared}, $G=L_{n,k}$.

The proof is complete.
\hfill{\rule{4pt}{8pt}}

\section{One open problem of Nikiforov }\label{Sec:Nikiforov}
This section is devoted to an open problem by Nikiforov \cite{N08}.
Before the proof, we collect various results that will be used in our arguments.

We first prove one edge condition for even cycles.

\begin{thm}\label{Thm:Erdos-Gall-Voss}
Let $G$ be a graph on $n$ vertices and $e(G)$ edges.
If $G$ contains no even cycle of length more than $2k$, where
$k\geq 1$ is an integer, then $e(G)\leq \frac{(2k+1)(n-1)}{2}$.
\end{thm}

\begin{thm}[Voss and Zuluaga \cite{VZ77}]\label{Thm:Voss-Zuluage}
(1) Every 2-connected graph $G$ with $\delta(G)\geq r\geq 3$ having at least $2r+1$
vertices contains an even cycle of length at least $2r$.
(2) Every 2-connected non-bipartite graph $G$ with $\delta(G)\geq r\geq 3$ having at least $2r+1$
vertices contains an odd cycle of length at least $2r-1$.
\end{thm}

\begin{thm}[Ore \cite{O61}]\label{Thm:Ore}
Let $G$ be a graph on $n$ vertices. If $G$ contains no Hamilton cycle, then
$e(G)\leq \binom{n-1}{2}+1$.
\end{thm}
A graph is called a \emph{theta graph} if it consists of three paths starting and
ending with two same vertices and are internal-disjoint. The following
lemma is very basic.
\begin{lem}\label{Lem:Theta}
Let $G$ be a graph containing no theta graphs. Then each component of $G$
is an edge or a cycle.
\end{lem}

\noindent
{\bf Proof of Theorem \ref{Thm:Erdos-Gall-Voss}.}
If $n\leq 2k+1$, then $e(G)\leq \binom{n}{2}\leq \frac{(2k+1)(n-1)}{2}$.
If $n=2k+2$, then by Theorem \ref{Thm:Ore}, we have $e(G)\leq \binom{2k+1}{2}+1\leq \frac{(2k+1)(n-1)}{2}$.
Next, we assume $n\geq 2k+3$.

Let $k=1$. We shall prove that if a graph on $n$ vertices contains
no even cycles then $e(G)\leq \frac{3(n-1)}{2}$.
By Lemma \ref{Lem:Theta}, every component of $G$ is an edge or an odd cycle. Let $c$
be the number of components which are odd cycles. We use induction to prove that
$e(G)\leq n+c-1\leq n-1+\frac{n-1}{2}=\frac{3(n-1)}{2}$.
In the following, we suppose $k\geq 2$.

Let $v\in V(G)$ with $d_G(v)=\delta(G)$, and $G':=G-v$. Note that $G'$ satisfies that
$v(G')\geq2k+2$ and $G'$ contains no even cycle of length more than $2k$.
By induction hypothesis, if $d(v)\leq k$, then we have
$e(G)=e(G')+\delta\leq \frac{(2k+1)(n-2)}{2}+k<\frac{(2k+1)(n-1)}{2}$,
as required. Thus, $\delta(G)\geq k+1\geq 3$. If $G$ is 2-connected,
then by Theorem \ref{Thm:Voss-Zuluage},  $G$ contains an even cycle
of length at least $2k+2$, a contradiction. Thus, $G$ contains a cut-vertex
or is disconnected. For each case, we use induction to each component
and compute the number of edges. The proof is complete. \hfill{\rule{4pt}{8pt}}

The following spectral inequality was originally proposed by Guo, Wang and Li \cite{GWL19}
as a conjecture and proved by Sun and Das \cite{SD20}.
\begin{thm}[Sun and Das \cite{SD20}]\label{Thm:Sun-Das}
Let $G$ be a graph with minimum degree $\delta(G)\geq 1$. For any $v\in V(G)$,
we have $\rho^2(G-v)\geq\rho^2(G)-2d(v)+1$.
\end{thm}

By Theorem \ref{Thm:Sun-Das}, we deduce a result for graphs with isolated vertices.
\begin{lem}\label{Lem:Sun-Das}
Let $G$ be a graph. For any $v\in V(G)$, we have
$$\rho^2(G)\leq \rho^2(G-v)+2d(v).$$
\end{lem}

For a graph $G$, denote by $ec(G)$ the length of a longest even cycle of $G$ and
$oc(G)$ the length of a longest odd cycle of $G$.
\begin{thm}[Gould, Haxell and Scott \cite{GHS02}]\label{Thm:GHS}
For every real number $c>0$, there exists a constant $K:=K(c)=\frac{7.5\times 10^5}{c^5}$ depending only on $c$
such that the following holds. Let $G$ be a graph with $n\geq \frac{45K}{c^4}$ vertices
and minimum degree at least $cn$. Then $G$ contains a cycle of length $t$
for every even $t\in [4,ec(G)-K]$ and every odd $t\in [K,oc(G)-K]$.
\end{thm}

Now we give the proof of Theorem \ref{Thm:SpectraConsecutiveCycles}.

\noindent
{\bf Proof of Theorem \ref{Thm:SpectraConsecutiveCycles}.}
If $G$ is disconnected, for example, $G$ contains 
$t$ components, then we can add $t-1$ edges to 
make it connected and 1-edge-connected, i.e., each
new edge is a edge-cut of the new graph $G'$. 
Note that $\rho(G')\geq \rho(G)$.
For any integer $k\geq 3$, $G$ contains a cycle 
of length $k$ if and only if $G'$ contains a cycle
of length $k$. Thus, we can assume that $G$ is connected.

By Theorem \ref{Thm:Hong}, we have
$$\frac{n^2-1}{4}\leq \rho^2(T_{n,2})<\rho^2(G)\leq
2m-n+1.$$ One can compute that
$2m\geq\frac{n^2+4n+3}{4}$. Thus, the average degree
$d(G):=\frac{2m}{n}>\frac{n}{4}$.

Let $H$ be a subgraph of $G$ defined by a sequence of graphs
$G_0,G_1,\ldots,G_k$ such that:\\
(1) $G=G_0$, $H=G_k$;\\
(2) for every $i\in[0,k-1]$, there is $v_i\in V(G_i)$ such that
$d_{G_i}(v_i)\leq\frac{n}{8}$ and $G_{i+1}=G_i-v_i$;\\
(3) for every $v\in V(G_k)$, $d_{G_k}(v)>\frac{n}{8}$.\\
We claim that $d(H)>\frac{n}{4}$. Suppose not the case. Then there
is a smallest $i\in[1,k]$ with $d(G_i)\leq\frac{n}{4}$. This implies
that
$$d(G_{i-1})=\frac{2d(v_{i-1})+|G_i|d(G_i)}{|G_i|+1}\leq\frac{n}{4},$$
a contradiction. Thus, we conclude that $d(H)>\frac{n}{4}$ and
$\delta(H)>\frac{n}{8}$.

\underline{Case A: Even cycle.} Note that
$e(H)=\frac{d(H)|H|}{2}>\frac{\frac{n}{4}(|H|-1)}{2}$. By Theorem
\ref{Thm:Erdos-Gall-Voss}, $ec(H)>\frac{n}{4}$. Recall that
$\delta(H)>\frac{n}{8}$. By Theorem \ref{Thm:GHS}, $H$ contains all
even cycles $C_{\ell}$ with $\ell \in [4,ec(G)-K]$ if
$|H|\geq45\cdot8^4\cdot K$, where $K=K(\frac{1}{8})=\frac{7.5\times 10^{5}}{(\frac{1}{8})^5}$ be the constant in
Theorem \ref{Thm:GHS}. Clearly $|H|>\frac{n}{4}$. Let $n_1$ be an
integer satisfying
$$(i)\ \frac{n_1}{4}\geq 45\cdot8^4\cdot K;\  (ii)\ \varepsilon n_1\geq K.$$
Now if $n\geq \max\{1.9\times 10^{16},\frac{2.5\times 10^{10}}{\varepsilon}\}$,
then $G$ contains all even cycles $C_{\ell}$
with $\ell \in [4,(\frac{1}{4}-\varepsilon)n]$.

\underline{Case B: Odd cycle.} Set $h=|H|$. By Lemma
\ref{Lem:Sun-Das}, we have
$$
\rho^2(G)\leq\rho^2(H)+2\sum_{i=0}^{k-1}d_{G_i}(v_i)\leq\rho^2(H)+2k\cdot\frac{n}{8}=\rho^2(H)+\frac{kn}{4},
$$
where $G_i$, $v_i$ are those in the definition of $H$, and $k=n-h$.
This implies that $\rho(H)\geq\frac{\sqrt{nh-1}}{2}$. Since
$\rho(G)>\sqrt{\lfloor\frac{n^2}{4}\rfloor}$, by Nosal's theorem \cite{N70}
and Mantel's theorem, $G$ contains a triangle, and so is
non-bipartite. If $h<n$, then
$\rho(H)>\frac{\sqrt{nh-1}}{2}\geq\sqrt{\lfloor\frac{h^2}{4}\rfloor}$,
and $H$ is non-bipartite as well. In any case we infer $H$ is
non-bipartite.

Let $F$ be a subgraph of $H$ defined by a sequence of graphs
$H_0,H_1,\ldots,H_k$ such that:\\
(1) $H=H_0$, $F=H_k$;\\
(2) for every $i\in[0,k-1]$, there is a cut-vertex $v_i$ of $H_i$ and $H_{i+1}=H_i-v_i$;\\
(3) $H_k$ has no cut-vertex.\\
Note that the component number $w(H_{i+1})\geq w(H_i)+1$. Clearly
$w(H)\leq 8$, for otherwise $H$ will have a vertex of degree less
than $\frac{n}{8}$. We claim that $w(F)\leq 8$. Suppose to the
contrary that there is a smallest $i$ with $w(H_i)\geq 9$. Notice that
$i\leq 8$, implying that $\delta(H_i)>\frac{n}{8}-8$. As $w(H_i)\geq
9$, $H_i$ has a vertex with degree less that
$\frac{|H_i|}{9}<\frac{n}{9}$, a contradiction when $n\geq577$. Thus
we conclude that $w(F)\leq 8$, and specially, $v(F)\geq h-7$.

By Lemma \ref{Lem:Sun-Das}, we have
$$
\rho^2(H)\leq\rho^2(F)+2\sum_{i=0}^{k-1}d_{H_i}(v_i)\leq\rho^2(F)+2k(h-1)\leq\rho^2(F)+14(h-1).
$$
Since $d(H)>\frac{n}{4}$, we obtain $h>\frac{n}{4}+1$. Thus,
\begin{equation*}
\begin{split}
\rho(F) & \geq\sqrt{\rho^2(H)-14(h-1)}\geq\sqrt{\frac{nh-1}{4}-14(h-1)}\\
        & =\sqrt{\left(\frac{n}{4}-14\right)h+\frac{55}{4}}>\sqrt{\left(\frac{n}{4}-14\right)\left(\frac{n}{4}+1\right)+\frac{55}{4}}\geq\frac{n}{4}-7\\
\end{split}
\end{equation*}
when $n\geq85$.

Recall that $F$ has no cut-vertex, i.e., every component of $F$ is
2-connected. Let $F_1$ be a component of $F$ with
$\rho(F_1)=\rho(F)$. Thus we have $\delta(F_1)\geq\frac{n}{8}-7$ and
$\rho(F_1)>\frac{n}{4}-7$. Specially $|F_1|>\frac{n}{4}-6$.

We claim that $\delta(F_1)\geq\frac{|F_1|}{8}$. Recall that
$\delta(H)\geq\frac{n}{8}\geq\frac{|H|}{8}$, we assume that $F_1\neq
H$. This implies that $F$ has a second component $F_2$. Since
$\delta(F)\geq\frac{n}{8}-k$, we have $|F_2|\geq\frac{n}{8}-k+1$
(here $k$ is that in definition of $F$). This implies that
$|F_1|\leq h-k-(\frac{n}{8}-k+1)<\frac{7h}{8}$. Thus
$\delta(F_1)\geq\frac{n}{8}-7\geq\frac{7n/8}{8}\geq\frac{|F_1|}{8}$
when $n\geq 448$.

Now we show that $F_1$ is non-bipartite. Recall that $H$ is
non-bipartite. So we assume that $F_1\neq H$. By the analysis above
we have $|F_1|<\frac{7h}{8}$. Thus
$$\rho^2(F_1)=\rho^2(F)\geq\left(\frac{n}{4}-14\right)h+\frac{55}{4}\geq\left(\frac{h}{4}-14\right)h+\frac{55}{4}>\frac{(7h/8)^2}{4}>\frac{|F_1|^2}{4}$$
when $h\geq 238$. Since $h>\frac{n}{4}+1$, we have that $F$ is
non-bipartite when $n\geq 944$.

By Theorem \ref{Thm:Voss-Zuluage}, $oc(F_1)\geq
2\delta(F_1)-1\geq\frac{n}{4}-15$. By Theorem \ref{Thm:GHS}, $F_1$
contains all odd cycles $C_{\ell}$ for
$\ell\in[K,\frac{n}{4}-15-K]$, where $K=K(\frac{1}{8})$ is the
constant as in Theorem \ref{Thm:GHS}. A theorem of
Nikiforov \cite[Theorem~1]{N08}\footnote{By refining the proof, one can let $N=8400$.}
states that there exists a sufficiently large $N$
such that any graph of order $n\geq N$ has a cycle of length $\ell$ for every
$\ell \in[3,\frac{n}{320}]$. Let
$n_2$ be an integer such that
$$(i)\ n_2\geq\max\{944,N\};\ (ii)\ \frac{n_2}{320}\geq K;\ (iii)\
\varepsilon n_2\geq K+15.$$
We only need $n_2\geq \max\{N, 7.9\times 10^{12},\frac{2.5\times 10^{10}}{\varepsilon}\}$.
Now if $n\geq\max\{n_1, n_2\}$, then $G$
contains all cycles $C_{\ell}$ with $\ell \in
[3,(\frac{1}{4}-\varepsilon)n]$.

The proof is complete. \hfill{\rule{4pt}{8pt}}

\section{A concluding  remark}

Nikiforov \cite{N10} proposed two nice conjectures on cycles of small lengths.
He conjectured that:  (a)  every graph on sufficiently large order $n$
contains a $C_{2k+1}$ or a $C_{2k+2}$ if $\rho(G)\geq \rho(S_{n,k})$,
unless $G=S_{n,k}$ where $S_{n,k}:=K_k\vee (n-k)K_1$; and (b) every graph on sufficiently large order $n$
contains a $C_{2k+2}$ if $\rho(G)\geq \rho(S^+_{n,k})$,
unless $G=S^+_{n,k}$ where $S^+_{n,k}$ is obtained from $S_{n,k}$
by adding an edge in the $n-k$ isolated vertices. One can easily compute
that $\rho(S_{n,k})=\Omega (\sqrt{n})$ and $\rho(S^+_{n,k})=\Omega (\sqrt{n})$.
If these conjectures will be confirmed, then we maybe obtain tight
spectral conditions for $C_{\ell}$ where $\ell \in [3,\Omega(\sqrt{n})]\cup [n-\Omega(\sqrt{n}),n]$.
It is mysterious to determine tight spectral conditions
for $C_{\ell}$, where $0<\lim\limits_{n\rightarrow \infty} \frac{\ell}{n}=c<1$,
such as $C_{\lfloor\frac{n}{2}\rfloor}$ and $C_{\lceil\frac{n}{2}\rceil}$
and etc.

\end{document}